\documentclass{amsart}
\usepackage{amssymb,amsthm,url}

\newtheorem{theorem}{Theorem}
\newtheorem{lemma}[theorem]{Lemma}
\newtheorem{corollary}[theorem]{Corollary}

\theoremstyle{definition}
\newtheorem{definition}[theorem]{Definition}
\newtheorem{problem}[theorem]{Problem}
\newtheorem{question}[theorem]{Question}

\newcommand{\Cons}{\mathop{\rm Cons}}
\newcommand{\Exp}{\mathop{\rm Exp}}
\newcommand{\Pol}{\mathop{\rm Poly}}  
\newcommand{\Subexp}{\mathop{\rm SubExp}}  
\newcommand{\Sub}{\mathop{\rm SubPoly}}  
\newcommand{\cC}{\mathcal C}
\newcommand{\cP}{\mathcal P}
 
\newcommand{\C}{\mathrm C}
\newcommand{\D}{\mathrm D}
\newcommand{\K}{\mathrm K}
\newcommand{\V}{\mathrm V}
\newcommand{\A}{\mathrm A}
\newcommand{\NN}{\mathbb N}
\newcommand{\RR}{\mathbb R}
\newcommand{\ZZ}{\mathbb Z}
\newcommand{\FF}{{\mathbb F}}
\renewcommand{\wr}{\mathop{\rm wr}}

\newcommand{\Aut}{\mathrm{Aut}}
\newcommand{\Sym}{\mathrm{S}}

\newcommand{\tv}{{\tilde{v}}}
\newcommand{\tu}{{\tilde{u}}}

\newcommand{\one}{{\bf{1}}}

\newcommand{\la}{\langle}
\newcommand{\ra}{\rangle}

\newcommand{\norml}{\trianglelefteq}

\makeatletter
\def\imod#1{\allowbreak\mkern10mu({\operator@font mod}\,\,#1)}
\makeatother

\begin{document}

\title[On the order of arc-stabilisers]{On the order of arc-stabilisers in arc-transitive graphs with prescribed local group}

\author[P. Poto\v{c}nik]{Primo\v{z} Poto\v{c}nik}
\address{Primo\v{z} Poto\v{c}nik,\newline
 Faculty of Mathematics and Physics,
 University of Ljubljana, \newline 
Jadranska 19, 1000 Ljubljana, Slovenia}\email{primoz.potocnik@fmf.uni-lj.si}

\author[P. Spiga]{Pablo Spiga}
\address{Pablo Spiga,\newline
 University of Milano-Bicocca, Departimento di Matematica Pura e Applicata, \newline
 Via Cozzi 53, 20126 Milano Italy} \email{pablo.spiga@unimib.it}

\author[G. Verret]{Gabriel Verret}
\address{Gabriel Verret,\newline
Faculty of Mathematics, Nat. Sci. and Info. Tech., University of Primorska, \newline 
Glagolja\v{s}ka 8, 6000 Koper, Slovenia}
\email{gabriel.verret@pint.upr.si}



\begin{abstract}
Let $\Gamma$ be a connected $G$-arc-transitive graph, let $uv$ be an arc of $\Gamma$ and let $L$ be the permutation group induced by the action of the vertex-stabiliser $G_v$ on the neighbourhood $\Gamma(v)$. We study the problem of bounding $|G_{uv}|$ in terms of $L$ and the order of $\Gamma$.
\end{abstract}

\maketitle

\section{Introduction}
All graphs considered in this paper are finite, simple and connected. A graph $\Gamma$ is said to be $G$-\emph{vertex-transitive} if $G$ is a subgroup of $\Aut(\Gamma)$ acting transitively on the vertex-set $\V(\Gamma)$ of $\Gamma$. Similarly, $\Gamma$ is said to be $G$-\emph{arc-transitive} if $G$ acts transitively on the arc-set $\A(\Gamma)$ of $\Gamma$ (an \emph{arc} is an ordered pair of adjacent vertices). 

For a vertex $v$ of $\Gamma$ and for $G\le \Aut(\Gamma)$, let $\Gamma(v)$ be the neighbourhood of $v$ in $\Gamma$ and let $G_v^{\Gamma(v)}$ be the permutation group induced by the action of the stabiliser $G_v$ on $\Gamma(v)$. We shall often refer to the group $G_v^{\Gamma(v)}$ as the {\em local group} of the pair $(\Gamma,G)$. Of course, if $G$ acts transitively on the arcs of $\Gamma$, then the local group $G_v^{\Gamma(v)}$ is transitive and (up to permutation isomorphism) independent of the choice of $v$. If $\Gamma$ is a $G$-arc-transitive graph and $L$ is a permutation group which is permutation isomorphic to $G_v^{\Gamma(v)}$, then we say that the pair $(\Gamma,G)$ is {\em locally-$L$}.

In~\cite{Verret}, the following notion was introduced: a transitive permutation group $L$ is called \emph{graph-restrictive} if there exists a constant $c(L)$ such that, for every locally-$L$ pair $(\Gamma,G)$ and for every arc $(u,v)$ of $\Gamma$, the inequality $|G_{uv}|\leq c(L)$ holds. Proving that certain permutation groups are graph-restrictive is a classical topic in algebraic graph theory; for example, the famous result of Tutte on cubic arc-transitive graphs \cite{Tutte,Tutte2} states that the symmetric group of degree $3$ is graph-restrictive and the still open conjectures of Weiss \cite{Weiss} and Praeger \cite{PConj} claim that every primitive as well as every quasiprimitive permutation group is graph-restrictive. The problem of determining which transitive permutation groups are graph-restrictive was proposed in~\cite{Verret}. A survey of the state of this problem can be found in~\cite{PSVRestrictive}.

There are several reasons why one might like to control the order of the arc-stabiliser $G_{uv}$ in a locally-$L$ pair $(\Gamma,G)$ even when the local group $L$ is not graph-restrictive. While $|G_{uv}|$ can be arbitrarily large in this case,  it would often suffice to obtain a good bound on $|G_{uv}|$ in terms of $|\V(\Gamma)|$; for example, if $|G_{uv}|$ can be bounded by a reasonably tame function of $|\V(\Gamma)|$, then the method described in \cite{ConDob} can be applied to obtain a complete list of all locally-$L$ pairs on a small number of vertices.

Bounding $|G_{uv}|$ in terms of $|\V(\Gamma)|$ and the local group $G_v^{\Gamma(v)}$ is precisely the goal that we pursue in this paper. In fact, it is not hard to see that there is always an exponential upper bound  on $|\Aut(\Gamma)_{uv}|$ in terms of $|\V(\Gamma)|$ (see Theorem~\ref{the:exp}). It is thus very natural to ask for which local groups a subexponential upper bound exists. This question motivates the following definition.

\begin{definition}
\label{def:f-restrictive} 
Let $f\colon  \NN \to \RR$ be a function and let $L$ be a transitive permutation group. If, for every locally-$L$ pair $(\Gamma,G)$  and every arc $(u,v)$ of $\Gamma$, the inequality $|G_{uv}|\leq f(|V(\Gamma)|)$ holds, then $L$ is called \emph{$f$-graph-restrictive}. On the other hand, if, for every integer $n$, there exists a locally-$L$ pair  $(\Gamma,G)$ with $|\V(\Gamma)| \ge n$ and $|G_{uv}| \ge f(|\V(\Gamma)|)$, then $L$ is called \emph{$f$-graph-unrestrictive}.
\end{definition}

Note that, for each transitive permutation group $L$, there exists a function $f$ such that $L$ is both $f$-graph-restrictive and $f$-graph-unrestrictive (for example, let $f(n)$ be the largest possible order of $G_{uv}$ in a locally-$L$ pair $(\Gamma,G)$ with $\Gamma$ having order $n$).  On the other hand, finding such a function explicitly is quite difficult in general. We therefore define graph-restrictiveness and graph-unrestrictiveness over a class of functions.

\begin{definition}
Let $\cC$ be a class of functions. If $L$ is $f$-graph-restrictive (respectively, $f$-graph-unrestrictive) for some function $f \in \cC$, then we say that $L$ is {\em $\cC$-graph-restrictive} (respectively, {\em $\cC$-graph-unrestrictive}). If $L$ is both $\cC$-graph-restrictive and $\cC$-graph-unrestrictive, then we say that $L$ has \emph{graph-type $\cC$}.
\end{definition} 

Given a transitive group $L$, we would like to find a ``natural" class of functions $\cC$ such that $L$ has graph-type $\cC$. The classes of functions that will particularly interest us are: the class $\Cons$ of constant functions, the class $\Pol$ of functions of the form $f(n) = n^\alpha$ for some $\alpha > 0$ and the class $\Exp$ of functions of the form $f(n)=\alpha^n$ for some $\alpha>1$. 

We also define the intermediate classes : the class $\Sub$ of functions of the form $f(n)$ such that $f(n)$ is unbounded and $\frac{\log(f(n))}{\log(n)}\to 0$ as $n\to \infty$ and the class $\Subexp$ of functions of the form $f(n)$ such that $\frac{\log(f(n))}{\log(n)}$ is unbounded and $\frac{\log(f(n))}{n}\to 0$ as $n\to \infty$. 
Every transitive permutation group is $\Exp$-graph-restrictive (see Theorem~\ref{the:exp}). It is then an elementary exercise in analysis to show that a transitive permutation group has graph-type exactly one of $\Cons$, $\Sub$, $\Pol$, $\Subexp$ or $\Exp$.

\begin{problem}\label{mainprob}
Given a transitive permutation group $L$, find $\cC$ in $\{\Cons,\Sub,\Pol,$ $\Subexp,\Exp\}$ such that $L$ has graph-type $\cC$.
\end{problem}

We now give a brief summary of the results which are proved in the rest of the paper. In Section~\ref{Expupper}, we show that every transitive permutation group is $\Exp$-graph-restrictive. In Section~\ref{sec:wreath}, we show that the imprimitive wreath product of two non-trivial transitive permutation groups is $\Exp$-graph-unrestrictive and hence has graph-type $\Exp$.

In Section~\ref{sec:TwoBlocks}, we consider a permutation group $L$ that is transitive and admits a system of imprimitivity consisting of two blocks $A$ and $B$ and show that $L$ is $\Pol$-graph-unrestrictive unless $L$ is regular. Moreover, we show that if the pointwise stabiliser of $A$ in $L$ is non-trivial, then $L$ is actually $\Exp$-graph-unrestrictive and thus has graph-type $\Exp$.

In Section~\ref{deg6}, we consider the imprimitive permutation groups of degree $6$ that do not admit a system of imprimitivity consisting of two blocks of size $3$. There exist five such groups up to permutation isomorphism. Two of them are imprimitive wreath products and hence have graph-type $\Exp$. We show that the remaining three groups are $\Subexp$-graph-unrestrictive. However, for none of these three groups we were able to decide whether it has graph-type $\Subexp$ or $\Exp$.

Finally, in Section~\ref{sec:smalldegree}, we apply these results to solve Problem~\ref{mainprob} for permutation groups of degree at most $7$, except for the three undecided cases mentioned in the previous paragraph.  It turns out that there is a unique transitive permutation group of degree at most $7$ with graph-type $\Pol$ (namely, the dihedral group of order $12$ in its natural action on six points) while there are many with graph-types $\Cons$ and $\Exp$. On the other hand, we know of no transitive permutation group with graph-type $\Sub$ or $\Subexp$.

\begin{question}
Does there exist a transitive permutation group with graph-type $\Sub$ or $\Subexp$?
\end{question}

\section{General exponential upper bound}
\label{Expupper}

In this section, we show that every transitive permutation group is $\Exp$-graph-restrictive.

\begin{theorem}
\label{the:exp}
Let $L$ be a transitive permutation group and let $L_\omega$ be a point-stabiliser in $L$. Then $L$ is $f$-graph-restrictive where $f(n) = |L_\omega|^{\frac{n-2}{2}}$. In particular, $L$ is $\Exp$-graph-restrictive.
\end{theorem}

\begin{proof}
Let $(\Gamma,G)$ be a locally-$L$ pair, let $n=|\V(\Gamma)|$ and let $(u,v)$ be an arc of $\Gamma$. Recall that a group $A$ is called a  {\em section} of a group $B$ provided that $A$ is isomorphic to a quotient of some subgroup of $B$.

We shall now recursively construct an increasing sequence of subsets $S_i$ of  $\V(\Gamma)$ and a decreasing subnormal sequence of subgroups $G_i$ of $G$,

\begin{equation}
\label{eq:SG}
 S_0 \subseteq S_1 \subseteq \cdots \subseteq S_m \quad \hbox{ and } \quad  G_0 \trianglerighteq G_1 \trianglerighteq \cdots \trianglerighteq G_m,
\end{equation}
such that the conditions (i)--(v) are fulfilled. (We use the notation $\Gamma[S]$ to denote the subgraph of a graph $\Gamma$ induced by a set of vertices $S \subseteq \V(\Gamma)$.)
\begin{itemize}
\item[{\rm (i)}]
$S_0 = \{u, v\}$ and $|S_i| \ge |S_{i-1}| + 2$ for every $i\in \{1,\ldots,m\}$,
\item[{\rm (ii)}]
$\Gamma[S_i]$ is connected for every $i\in \{0,\ldots,m\}$,
\item[{\rm (iii)}]
$G_i = G_{(S_i)}$ for every $i\in \{0,\ldots,m\}$,
\item[{\rm (iv)}]
$G_{i-1}/G_i$ is a section of $L_\omega$ for every $i\in \{1,\ldots,m\}$,
\item[{\rm (v)}]
$G_m = 1$.
\end{itemize}

Let $S_0 = \{u,v\}$ and let $G_0 =G_{(S_0)} = G_{uv}$. Clearly conditions (i)--(iv) are satisfied with $m=1$. Suppose now that for some $k\ge 0$ the sets $S_0, \ldots, S_k$ and the groups $G_0, \ldots, G_k$ are defined and that they satisfy the conditions (i)--(iv) with $m=k$.

Let $S$ be the set of all vertices of $\Gamma$ that are fixed by $G_k$ and let $S'$ be the vertex-set of the connected component of $\Gamma[S]$ that contains $S_k$. Then clearly $S_k\subseteq S'\subseteq S$ hence $G_k\leq G_{(S)}\leq G_{(S')}\leq G_{(S_k)}=G_k$ and $G_{(S')} = G_k$.

If $S' = \V(\Gamma)$, then we let $m=k$ and terminate the construction. Observe that in this case the group $G_m$ is trivial, as required by condition (v).

If $S'$ is a proper subset of $\V(\Gamma)$, then it follows from the definition of $S'$ that there exists a vertex in $\V(\Gamma)\setminus S'$, say $x$, which is adjacent to some vertex, say $w$, in  $S'$ and which is not fixed by $G_k$. Moreover, since $\Gamma[S']$ is  connected and contains at least two vertices, there exists a neighbour of $w$, say $z$, which is contained in $S'$. 

Let $X=x^{G_k}$, let $S_{k+1} = S' \cup X$ and let $G_{k+1} = G_{(S_{k+1})}$. We need to show that the extended sequences $(S_i)_i$ and $(G_i)_i$ still satisfy the conditions (i)--(iv) with $m=k+1$. Indeed, since $x$ is not fixed by $G_k$, we see that $|X|\ge 2$ and thus condition (i) holds. Similarly, conditions (ii) and (iii) hold by construction of $S_{k+1}$ and $G_{k+1}$. To show (iv), observe that $G_{k+1}$ is the kernel of the action of $G_k$ on $X$. Hence  $G_{k}/G_{k+1}$ is isomorphic to the permutation group $G_k^X$ induced by the action of $G_k$ on $X$. However, $G_k^X$ can also be viewed as the permutation group induced by the action of $G_k^{\Gamma(w)}$ on $X$, and is thus isomorphic to a quotient of $G_k^{\Gamma(w)}$. Since $G_k \le G_{wz}$, it follows that $G_k^X$ 
is a section of $G_{wz}^{\Gamma(w)}$. Since the latter group is isomorphic to $L_\omega$, this shows that condition (iv) holds as well.

The result of the above construction is thus a pair of sequences (\ref{eq:SG}) satisfying conditions (i)--(v). Now observe that condition (i) implies that $m\le \frac{n-2}{2}$. On the other hand, condition (iv) implies that $|G_i/G_{i+1}| \le |L_\omega|$ for every $i\in \{1,\ldots,m\}$. In view of condition (v), this implies that $|G_0| \le |L_\omega|^m$. Since $G_0 = G_{uv}$, this completes the proof.
\end{proof}

\noindent\textbf{Remark.}
The upper bound provided by the function $f$ in Theorem~\ref{the:exp} is rather crude and can be improved if some further information about the permutation group $L$ is taken into consideration. For example, if $p$ is the smallest prime dividing $|L_\omega|$, then the orbit $X$ introduced in the proof of Theorem~\ref{the:exp} is of length at least $p$ and thus condition (i) can be replaced by $|S_{i}| \ge |S_{i+1}| + p$. The definition of $f$ in Theorem~\ref{the:exp} can then be replaced by $f(n)=|L_\omega|^\frac{n-2}{p}$.

\section{Imprimitive wreath products}
\label{sec:wreath}
In this section, we show that the imprimitive wreath product of two non-trivial transitive permutation groups is $\Exp$-graph-unrestrictive. We first need the following lemma.

\begin{lemma}
\label{homologicalLemma}
Let $T$ be a transitive permutation group and let $T_\omega$ be a point-stabiliser in $T$. Then there exists a sequence of locally-$T$ pairs $(\Gamma_i,H_i)$, with
$|V(\Gamma_i)| \to \infty$ as $i\to \infty$,
such that, for every $i\geq 1$, the stabiliser of an arc of $\Gamma_i$ in $H_i$ has order $|T_\omega|^2$. 
\end{lemma}
\begin{proof}
There are several ways how to prove the existence of such a sequence. We construct it recursively using the theory of covering projections of graphs. We refer the reader to \cite{MMP} for further information on this topic. In particular, we refer the reader to \cite[Section 6]{MMP} for the definition of a homological $p$-cover of a graph.

Let $k$ be the degree of $T$, let $\Gamma_1 = \K_{k,k}$ be the complete bipartite graph with bipartition sets of size $k$, and let $H_1 = (T\times T) \rtimes\Sym_2$ acting arc-transitively on $\Gamma_1$ in the natural way. Note that $(\Gamma_1,H_1)$ is locally-$T$ and an arc-stabiliser has order $|T_\omega|^2$.

Suppose now that $(\Gamma_i,H_i)$ has already been constructed for some $i\geq 1$. Let $\Gamma_{i+1}$ be a homological $2$-cover of $\Gamma_i$ and let $H_{i+1}$ be the lift of $H_i$ along the covering projection $\Gamma_{i+1} \to \Gamma_i$ (note that by \cite[Proposition 6.4]{MMP} the group $H_i$ indeed lifts along this covering projection). Note that the vertex-stabilisers in $H_{i+i}$ and in $H_i$ are isomorphic and, moreover, that they induce permutation isomorphic groups on the respective neighbourhoods. In particular,
 the pair $(\Gamma_{i+1}, H_{i+1})$ is locally-$T$ and an arc-stabiliser has order $|T_\omega|^2$.
 \end{proof}

\begin{theorem}
\label{thm:wreath}
Let $R$ and $T$ be non-trivial transitive permutation groups, let $m$ be the degree of $R$ and let $T_\omega$ be a point-stabiliser in $T$. Then the imprimitive wreath product $R \wr T$ is $f$-graph-unrestrictive where $f(n) =\frac{|T_\omega|^2|R|^{\frac{n}{m}}}{m^2} $. In particular, $R \wr T$ has graph-type $\Exp$.
\end{theorem}

\begin{proof}
Let $\Delta$ and $\Omega$ be the sets on which $R$ and $T$ act, respectively. Then $R \wr T$ is a permutation group on $\Delta \times \Omega$. Writing $\Omega = \{1,2, \ldots, k\}$ yields that $R\wr T$ is isomorphic to the semidirect product $R^k\rtimes T$ where the action of
 $(a_1, \ldots, a_k) \in R^k$ and $b\in T$ on $\Delta \times \Omega$ is given by:
$$
 (\delta,\omega)^{(a_1,\ldots,a_k)} = (\delta^{a_\omega},\omega) \quad \hbox{ and }  \quad
 (\delta,\omega)^{b} = (\delta,\omega^b). 
$$

By Lemma~\ref{homologicalLemma}, there exists a sequence of locally-$T$ pairs $(\Gamma_i,H_i)$, with $|V(\Gamma_i)| \to \infty$, such that, for every $i\geq 1$, the stabiliser of an arc of $\Gamma_i$ in $H_i$ has order $|T_\omega|^2$. Let $\Lambda_i$ be the lexicographic product of $\Gamma_i$ with the edgeless graph on the vertex-set $\Delta$,  that is, the graph with vertex-set $\Delta \times \V(\Gamma_i)$ and two vertices $(\delta_1,v_1)$ and  $(\delta_2,v_2)$ adjacent in $\Lambda_i$ whenever $v_1$ and $v_2$ are adjacent in $\Gamma$.
Note that $|\V(\Lambda_i)|=m|\V(\Gamma_i)|$ and that $|\A(\Lambda_i)|=m^2 |\A(\Gamma_i)|$. 

Observe that the imprimitive wreath product $G_i=R\wr H_i$ acts on $\Lambda_i$ as an arc-transitive group of automorphisms and that the local group $(G_i)_v^{\Lambda_i(v)}$ is permutation isomorphic to $R \wr T$.

Let $(u,v)$ be an arc of $\Gamma_i$ and let $(\tu,\tv)$ be an arc of $\Lambda_i$. Since $\Gamma_i$ is $H_i$-arc-transitive and $\Lambda_i$ is $G_i$-arc-transitive, it follows that:
\begin{equation*}
\label{eqn:Gtutv}
 |(G_i)_{\tu\tv}| = \frac{|G_i|}{|\A(\Lambda_i)|}    = \frac{|R|^{|\V(\Gamma_i)|} |H_i|}{m^2 |\A(\Gamma_i)|}= \frac{|(H_i)_{uv}||R|^{\frac{|\V(\Lambda_i)|}{m}}}{m^2}
 =\frac{|T_\omega|^2|R|^{\frac{|\V(\Lambda_i)|}{m}}}{m^2}.
\end{equation*}
Since $|\V(\Lambda_i)| \to  \infty$ as $i\to \infty$, the result follows.
\end{proof}

\section{System of imprimitivity consisting of two blocks}
\label{sec:TwoBlocks}
In this section, we consider a permutation group $L$ that is transitive and admits a system of imprimitivity consisting of two blocks $A$ and $B$ and show that $L$ is $\Pol$-graph-unrestrictive unless it is regular. Moreover, we show that if the pointwise stabiliser of $A$ in $L$ is non-trivial, then $L$ is actually $\Exp$-graph-unrestrictive.

To do this, we must first define the graphs $\C(k,r,s)$, which were first defined by Praeger and Xu~\cite{PraegerXu}.  Let $r$ and $s$ be positive integers with $r\geq 3$ and $1\leq s\leq r-1$. Let $\C(k,r,1)$ be the lexicographic product $\C_r[k\K_1]$ of a cycle of length $r$ and an edgeless graph on $k$ vertices. In other words, $\V(\C(k,r,1))=\ZZ_k\times\ZZ_r$ with $(u,i)$ being adjacent to $(v,j)$ if and only if $i-j \in \{-1,1\}$. A path in $\C(k,r,1)$ is called {\em traversing} if it contains at most one vertex from $\ZZ_k\times\{y\}$ for each $y\in\ZZ_r$. For $s\geq 2$, let $\C(k,r,s)$ be the graph with vertices being the traversing paths in $\C(k,r,1)$ of length $s-1$ and with two such $(s-1)$-paths being adjacent in $\C(k,r,s)$ if and only if their union is a traversing path in $\C(k,r,1)$ of length $s$.

 Clearly, $\C(k,r,s)$ is a connected $2k$-valent graph with $rk^s$ vertices. There is an obvious action of the wreath product $\Sym_k\wr\D_r$ as a group of automorphisms of $\C(k,r,1)$. (We denote the symmetric group and the dihedral group in their natural action on $n$ points by $\Sym_n$ and $\D_n$, respectively.) Moreover, every automorphism of $\C(k,r,1)$ has a natural induced action as an automorphism of $\C(k,r,s)$. We use these graphs to prove the following result.

\begin{theorem}\label{theo:cons}
Let $L$ be a transitive permutation group of degree $2k$ with a system of imprimitivity consisting of two blocks $A$ and $B$. Let $L_\omega$ be a point-stabiliser and let $L_{(A)}$ be the pointwise stabiliser of $A$ in $L$. Let $\ell\ge 1$ and $m\ge 2$ be integers, let $\Gamma = \C(k,2\ell m,m-1)$ and let $uv$ be an arc of $\Gamma$. Then there exists a group $G \le \Aut(\Gamma)$ such that $(\Gamma,G)$ is locally-$L$ and $|G_{uv}| =|L_{(A)}|^{2 m (\ell -1)} |L_\omega|^m$.
\end{theorem}

\begin{proof}
Let $K$ be the kernel of the action of $L$ on $\{A,B\}$. Clearly, $K$ is a normal subgroup of index $2$ in $L$. Fix an element $h\in L\setminus K$ and observe that $A^h=B$ and $B^h=A$. We can label the points of $A$ by $\ZZ_k\times\{0\}$ and the points of $B$ by $\ZZ_k\times\{1\}$ in such a way that $(x,0)^h=(x,1)$ for every $x\in\ZZ_k$. With respect to this fixed labeling, we can view $L$ as a subgroup of $\Sym_k\wr\Sym_2=(\Sym_k\times \Sym_k)\rtimes \Sym_2$, with $K\le \Sym_k\times \Sym_k$, where, for every $(a,b)\in K \le \Sym_k\times \Sym_k$ and every $x\in\ZZ_k$, we have 
\begin{equation}
(x,0)^{(a,b)}=(x^a,0)\>  \hbox{ and } \> (x,1)^{(a,b)}=(x^b,1).
\end{equation}

Let $r\in \Sym_k$ be such that $(x,1)^h = (x^r,0)$ for every $x\in \ZZ_k$. Then $(x,0)^{h^2} = (x,1)^h = (x^r,0)$ and $(x,1)^{h^2} = (x^r,0)^h = (x^r,1)$,
implying that 
\begin{equation}\label{eq:h2}
h^2=(r,r).
\end{equation}

To summarise, the action of $h$ on $A \cup B$ is given by:
\begin{equation}
\label{eqn:h}
 (x,0)^h = (x,1) \quad \hbox{ and } \quad (x,1)^h = (x^r,0).
\end{equation}

If $(a,b)\in K$ and $x\in\ZZ_k$, then $(x,1)^{(a,b)^h}=(x,0)^{(a,b)h}=(x^a,0)^h=(x^a,1)$ and $(x,0)^{(a,b)^h}=(x,0)^{h(a,b)^{h^2}h^{-1}}=(x,1)^{(a^r,b^r)h^{-1}}=(x^{b^r},1)^{h^{-1}}=(x^{b^r},0)$. This shows that
\begin{equation}
\label{eq:abh}
(a,b)^h=(b^r,a).
\end{equation}

Let $\Lambda=\C(k,2\ell m,1)$. Recall that $\V(\Lambda)=\ZZ_k\times \ZZ_{2\ell m}$ and that $\Sym_k \wr \D_{2\ell m}$ acts naturally as a group of automorphisms of $\Lambda$. We now define some permutations of $\V(\Lambda)$. Let $(x,y)\in \ZZ_k\times \ZZ_{2\ell m}$ be a vertex of $\Lambda$, let $(x,y)^s=(x,y+1)$ and let $(x,y)^t=(x,-y)$. Moreover, for any $c\in\Sym_k$ and any $i\in \ZZ_{2\ell m}$, let 

\begin{displaymath}
   (x,y)^{[c]_i} = \left\{
     \begin{array}{ll}
       (x^c,y) & {\rm if}\; y = i ,\\
       (x,y) & {\rm otherwise.}
     \end{array}
   \right.
\end{displaymath}  

Note that for every $c,d \in \Sym_k$ and every $i, j \in \ZZ_{2\ell m}$, we have
\begin{equation}
\label{eqn:new}
  [c]_i [d]_i = [cd]_i \> \hbox { and } \> [c]_i [d]_j = [d]_j [c]_i \hbox { if } i\not = j.
\end{equation}

Clearly, $s$, $t$ and $[c]_i$ are automorphisms of $\Lambda$. Note that $t^2=1$ and $tst=s^{-1}$. Moreover, for every $i,j \in \ZZ_{2\ell m}$,
we have  
\begin{equation}
\label{eqn:cisj}
[c]_i^{s^j}=[c]_{i+j} \> \hbox{ and } \> [c]_i^t=[c]_{-i}.
\end{equation}

A typical element of $L_{(A)}$ can be written in the form $(1,b)$ with $b\in \Sym_k$, and moreover, an element $(a,b) \in K$ belongs to $L_{(A)}$ if and only if $a=1$. Now let
\begin{equation}
\label{eqn:M0}
M_0 = \{ [b]_0 \mid (1,b) \in K\}
\end{equation}
and note that $M_0 \le \Aut(\Lambda)$. For $a\in \Sym_k$ and $i\in \ZZ_{2\ell m}$, let $\chi(a,i)$ be the automorphism of $\Lambda$ defined by
 \begin{equation}
 \label{chidef}
 \chi(a,i) = [a]_i \, [a]_{i+2m} \, [a]_{i+4m}\, \ldots\, [a]_{i+(2\ell-2)m}.
 \end{equation}
 
Note that for every $a,b\in \Sym_k$ the following holds:
\begin{equation}
\label{eqn:chi}
\begin{array}{l}
(\ref{eqn:chi}.1)\qquad \chi(a,i) = \chi(a,j) \hbox{ whenever } i \equiv j \hbox{ mod } 2m, \\
(\ref{eqn:chi}.2)\qquad\chi(a,i)\chi(b,i) = \chi(ab,i), \\
(\ref{eqn:chi}.3)\qquad\chi(a,i) \hbox{ and } \chi(b,j) \hbox{ commute whenever } i \not \equiv j \hbox{ mod } 2m, \\
(\ref{eqn:chi}.4)\qquad\chi(a,i)^s = \chi(a,i+1), \\
(\ref{eqn:chi}.5)\qquad\chi(a,i)^t = \chi(a,-i).
\end{array}
\end{equation}

Let 
$$N_0=\{\chi(a,0)\chi(b,m) \mid (a,b)\in K\}.$$
Using (\ref{eqn:chi}.2) and (\ref{eqn:chi}.3), one can see that $N_0$ is a subgroup of $\Aut(\Lambda)$. Recall that $r$ is the element of $\Sym_k$ such that $h^2 = (r,r)$ and let 
$$\sigma=\chi(r,-1)s.$$

Using~(\ref{eqn:chi}.3) and (\ref{eqn:chi}.4), it is easy to see that, for $i\in \{1,\ldots,2m\}$, we have

\begin{equation}
\label{eq:sigmapower}
\sigma^i = \chi(r,-i)\ldots \chi(r,-2)\chi(r,-1)s^i. 
\end{equation}

For any integer $i$, let 
\begin{equation}
\label{eq:Mi}
M_i=(M_0)^{\sigma^i}\> \hbox{ and } \> N_i=(N_0)^{\sigma^i}.
\end{equation}

We will now show that for every element $(a,b) \in K$ and  for every $j\in \ZZ_{2\ell m}$, we have:
\begin{equation}
\label{eqn:new2}
 (M_0)^{[a]_j}=(M_0)^{[b]_j}= M_0.
\end{equation}

Indeed, by definition, a typical element of $M_0$ is of the form $[\beta]_0$  such that $(1,\beta)\in K$. If $j\not = 0$, then by (\ref{eqn:new})  $[b]_j$ commutes with $[\beta]_0$ and hence centralises $M_0$. Suppose now that $j=0$. Then, by (\ref{eqn:new}), we have $[\beta]_0^{[b]_0} = [\beta^b]_0$. Since $(1,\beta)$ and $(a,b)$ are elements of $K$ and since $(1,\beta)^{(a,b)} = (1,\beta^b)$, we see that $(1,\beta^b)\in K$, which, by definition of $M_0$, implies that $[\beta^b]_0$ (and thus $[\beta]_0^{[b]_0}$) is in $M_0$. This shows that $[b]_j$ normalises $M_0$ for every $j$. Now recall that, by (\ref{eq:abh}), we have $(a,b)^h = (b^r,a)$. Since $h$ normalises $K$, this shows that $(b^r,a)\in K$. By applying the above argument with $(b^r,a)$ in place of $(a,b)$, we see that also $[a]_j$ normalises $M_0$, thus proving (\ref{eqn:new2}). 

By (\ref{chidef}) and (\ref{eqn:new2}),  for every element $(a,b) \in K$ and  for every $j\in \ZZ_{2\ell m}$, we have:
\begin{equation}
\label{eqn:new3}
(M_0)^{\chi(a,j)}=(M_0)^{\chi(b,j)} = M_0.
\end{equation}
 
 In particular, since since $(r,r)\in K$, the group $M_0$ is normalised by $\chi(r,j)$.
Together with (\ref{eqn:cisj}), (\ref{eq:sigmapower}) and (\ref{eq:Mi}), the latter implies that:
 \begin{equation}
 \label{eqn:Mis}
 M_i= (M_0)^{s^i} = \{ [\beta]_i \mid (1,\beta) \in K\}.
 \end{equation}

Let us now show that the following holds for every $(a,b)\in K$ and $i, j \in \ZZ_{2\ell m}$:
 \begin{equation}
 \label{eqn:Michi}
 (M_i)^{\chi(a,j)}=(M_i)^{\chi(b,j)}=(M_i)^{\chi(r,j)} = M_i.
\end{equation}

Indeed:
 \begin{eqnarray*}
 (M_i)^{\chi(a,j)} & {{(\ref{eqn:Mis})}\atop{=}} &  (M_0)^{s^i \chi(a,j)} \\
                             & {{(\ref{eqn:chi}.4)} \atop{=}} &   (M_0)^{ \chi(a,j-i) s^i} \\
                              & {{(\ref{eqn:new3})} \atop{=}} & (M_0)^{ s^i} \\
                               & {{(\ref{eqn:Mis})}\atop{=}} & M_i,
 \end{eqnarray*}
 and similarly,  $(M_i)^{\chi(b,j)} = M_i$, thus proving (\ref{eqn:Michi}).
\medskip

Furthermore, the definition (\ref{eq:Mi}) of $N_i$, (\ref{eqn:chi}.3) and (\ref{eq:sigmapower}) imply that
\begin{equation}
\label{eqn:Ki}
N_i=\{\chi(a,i)\chi(b,i+m)\mid (a,b)\in K\} \hbox{ for every } i\in \{1,\ldots,m-1\}.
\end{equation} 

We shall now show that $N_{m+i} = N_i$. Let $(a,b)\in K$ and observe that:
\begin{eqnarray*}
\bigl(\chi(a,0)\chi(b,m)\bigr) ^{\sigma^m} 
&{(\ref{eq:sigmapower})}\atop{=}&
\bigl(\chi(a,0)\chi(b,m)\bigr)^{\chi(r,-m)\ldots \chi(r,-2)\chi(r,-1)s^m}\\
&{(\ref{eqn:chi}.3)}\atop{=}&
\bigl(\chi(a,0)\chi(b,m)\bigr)^{\chi(r,-m)s^m}\\
&{(\ref{eqn:chi}.1,\, \ref{eqn:chi}.2)}\atop{=}&
\bigl(\chi(a,0)\chi(b^r,m)\bigr)^{s^m}\\
&{(\ref{eqn:chi}.3,\, \ref{eqn:chi}.4)}\atop{=}&
\chi(b^r,0)\chi(a,m).
\end{eqnarray*}

Since $(a,b)\in K$, by~$(\ref{eq:abh})$ we have $(a,b)^h=(b^r,a)\in K$ and hence $\chi(b^r,0)\chi(a,m)\in N_0$. This shows that $\sigma^m$ normalises $N_0$. However, $N_m = (N_0)^{\sigma^m}$ by definition, implying that $N_m = N_0$, and thus 
\begin{equation}
\label{eqn:Kim}
N_{i+m}=N_i \hbox{ for every integer } i.
\end{equation}

Let us consider the automorphism $\tau$ of $\Lambda$ defined by
\begin{equation}
\label{eqn:tau}
\tau=\chi(r,m+1)\chi(r,m+2) \ldots \chi(r,2m-1)t.
\end{equation}

Note that (\ref{eqn:cisj}) and (\ref{eqn:Mis}) imply
\begin{equation}
\label{eqn:Mit}
(M_i)^\tau= M_{-i} \hbox{ for every integer } i.
\end{equation}

We shall now show that 
\begin{equation}
\label{eqn:Ki-i}
(N_i)^\tau= N_{-i} \hbox{ for every integer } i.
\end{equation}

It follows from (\ref{eqn:chi}.3) that $\tau$ centralises $N_0$ and hence (\ref{eqn:Ki-i}) holds for $i=0$. By (\ref{eqn:Kim}), it thus suffices to show (\ref{eqn:Ki-i}) for  $i\in \{1,\ldots,m-1\}$. By (\ref{eqn:Ki}), a typical element of $N_i$ is of the form $\chi(a,i)\chi(b,m+i)$ for some $(a,b)\in K$. Then

\begin{eqnarray*}
\bigl(\chi(a,i)\chi(b,m+i)\bigr)^\tau
&(\ref{eqn:tau})\atop{=}& \bigl(\chi(a,i)\chi(b,m+i)\bigr)^{\chi(r,m+1)\chi(r,m+2)\ldots \chi(r,2m-1)t}\\
&{(\ref{eqn:chi}.2,\, \ref{eqn:chi}.3)}\atop{=}&
\bigl( \chi(b^r,m+i)\chi(a,i)\bigr)^t\\
&{(\ref{eqn:chi}.1,\, \ref{eqn:chi}.5)}\atop{=}&
\chi(b^r,m-i)\chi(a,-i).
\end{eqnarray*} 

Since $(a,b)\in K$,  by (\ref{eq:abh}) we have $(b^r,a)\in K$ and hence $\chi(b^r,m-i)\chi(a,{-i})\in N_{m-i}$.  This shows that $(N_i)^\tau=N_{m-i}$, which, by (\ref{eqn:Kim}), equals $N_{-i}$. This completes the proof of (\ref{eqn:Ki-i}).

Let $M$ ($N$, respectively) be the group generated by all $M_i$ ($N_i$, respectively), $i\in \ZZ$. By (\ref{eqn:Mis}) it follows that $M_i = M_{i+2\ell m}$. From this and from (\ref{eqn:Ki}), we deduce that $M = \langle M_0, \ldots, M_{2\ell m-1} \rangle$ and $N=\langle N_0,\ldots,N_{m-1}\rangle$. Further, by (\ref{eqn:chi}.3), (\ref{eqn:Ki}) and (\ref{eqn:Kim}), it follows that
\begin{equation}
\label{eq:N}
M =M_0\times\cdots\times M_{2\ell m-1} \hbox{ and } N=N_0\times\cdots\times N_{m-1}.
\end{equation}

Let $G=\langle M,N,\sigma,\tau\rangle$. By (\ref{eqn:Mit}) and (\ref{eqn:Ki-i}), it follows that  $M$ and $N$ are normalised by $\tau$. By definition, they are also normalised by $\sigma$, implying that $\langle M, N\rangle$ is normal in $G$.  Recall that $\V(\Lambda)$ admits a natural partition $\cP=\{\ZZ_k\times\{y\}\mid y\in\ZZ_{2\ell m}\}$ and observe that $\cP$ is in fact the set of orbits of $\langle M,N\rangle$ on $\V(\Lambda)$ and hence is $G$-invariant.

By the definitions of $\sigma$ and $\tau$, it follows that the permutation group induced by the action of $G$ on $\cP$ is isomorphic to the dihedral group $\D_{2\ell m}$ in its natural action of degree $2\ell m$.

We will now show that $\langle M,N\rangle$ is the kernel of the action of $G$ on $\cP$. Observe that it suffices to show that  $G/\langle M,N\rangle$ acts faithfully on $\cP$ or, equivalently, that $G/\langle M,N\rangle$ is isomorphic to $\D_{2\ell m}$. We will show this by proving that $\sigma^{2\ell m}$, $\tau^2$ and $\sigma\sigma^\tau$ are all contained in $\langle M,N\rangle$. We begin by computing $\sigma\sigma^\tau$. Observe that by  (\ref{eqn:chi}.3) and (\ref{eqn:chi}.4) we have $s^{\chi(a,i)} = \chi(a,i-1)\chi(a^{-1},i)s$. By applying this repeatedly and using (\ref{eqn:chi}.1), (\ref{eqn:chi}.2) and (\ref{eqn:chi}.3) we deduce that
$$s^{\chi(r,m+1)\chi(r,m+2)\cdots\chi(r,2m-1)} = \chi(r,m)\chi(r^{-1},-1)s,$$ 
and therefore 
\begin{eqnarray*}
\sigma^\tau&=& (\chi(r,-1)s)^{\chi(r,m+1)\chi(r,m+2)\cdots\chi(r,2m-1)t}\\
&(\ref{eqn:chi}.2, \ref{eqn:chi}.3)\atop=& (\chi(r,-1)\chi(r,m)\chi(r^{-1},-1)s)^t\\
&(\ref{eqn:chi}.1, \ref{eqn:chi}.5)\atop=&\chi(r,m)s^{-1}.
\end{eqnarray*} 

Computing $\sigma\sigma^\tau$ is now easy:
\begin{equation}
\label{eqn:sst}
\sigma\sigma^\tau \> = \> \chi(r,-1)s\chi(r,m)s^{-1} \> = \> \chi(r,-1)\chi(r,m-1).
\end{equation}

By~$(\ref{eq:h2})$, we know that $h^2=(r,r)\in K$ and hence, by (\ref{eqn:Ki}), $\chi(r,i)\chi(r,m+i)\in N_i$ and, in particular, $\sigma\sigma^\tau\in N$. Now,

\begin{eqnarray*}
\tau^2&=& \bigl(\chi(r,m+1)\cdots \chi(r,2m-1)t\bigr)^2\\
&(\ref{eqn:chi}.5)\atop=&\chi(r,m+1)\cdots \chi(r,2m-1)\chi(r,1)\cdots \chi(r,m-1)\\
&(\ref{eqn:chi}.3)\atop=&(\chi(r,1)\chi(r,m+1))(\chi(r,2)\chi(r,m+2))\cdots (\chi(r,m-1)\chi(r,2m-1)).\\
\end{eqnarray*} 

Again, $\chi(r,i)\chi(r,m+i)\in N_i$ for all $i\in\ZZ_{2\ell m}$ and hence $\tau^2\in N$. Finally, by~$(\ref{eq:sigmapower})$, we have 

\begin{eqnarray*}
\sigma^{2m}&=& \chi(r,0)\chi(r,1)\cdots \chi(r,2m-1)\\
&=&(\chi(r,0)\chi(r,m))(\chi(r,1)\chi(r,m+1))\cdots (\chi(r,m-1)\chi(r,2m-1))\in N.\\
\end{eqnarray*} 

We have just shown that $\sigma^{2m}$, $\tau^2$ and $\sigma\sigma^\tau$ are all contained in $N$, and therefore, that $\la M, N\ra$ is the kernel of the action of $G$ on $\cP$, as claimed. 

We shall now determine the order of $\la M, N\ra$. Note that by (\ref{eqn:Michi}), (\ref{eqn:Ki}) and (\ref{eq:N}),  the group $N$ normalises $M$,
implying that  $\la M, N\ra=MN$. Let us first determine the order of $M\cap N$.

Let $\alpha$ be an arbitrary element of $M\cap N$. Since $\alpha \in N$, we can write
\begin{equation*}
 \alpha = \prod_{j=0}^{m-1}{\chi(a_j,j)}\chi(b_j,j+m)
\end{equation*}
for some permutations $a_0, a_1, \ldots, a_{m-1}$, $b_{0}, b_{1}, \ldots, b_{m-1}$ of $\ZZ_k$ such that $(a_{j},b_{j}) \in K$ for all $j\in\{0,\ldots, m-1\}$. If we write $a_{j+m} = b_{j}$, we obtain
\begin{equation}
\label{eqn:alpha}
 \alpha = \prod_{j=0}^{2m-1}\chi(a_j,j)
\end{equation}
with $(a_j,a_{j+m}) \in K$ for all $j\in\{0,\ldots, m-1\}$. Now observe that for every $i\in \ZZ_{2\ell m}$, the set $\ZZ_k\times \{i\}  \subseteq \V(\Lambda)$ is preserved by $\alpha$. Furthermore, for every $j\in \{0,\ldots,2m-1\}$, the restriction of $\alpha$ to $\ZZ_k \times \{j\} $ (when viewed as a permutation on $\ZZ_k$) is in fact $a_{j}$. On the other hand, since $\alpha \in M$, such a restriction $a_{j}$ has to satisfy the defining property $(1,a_{j}) \in K$.

Conversely, let $a_{0}, a_{1}, \ldots, a_{2m-1}$ be arbitrary permutations of $\ZZ_k$ such that $(1,a_{j}) \in K$ for all $j\in \{0,\ldots,2m-1\}$ and let $\alpha$ be as in (\ref{eqn:alpha}). We shall now show that $\alpha$  is an element of $M\cap N$. 

The fact that $\alpha$ is in $M$ follows directly from (\ref{eqn:Mis}) and from the fact that $(1,a_j) \in K$ for each $j$. In order to prove that $\alpha\in N$, it suffices to show that $(a_{j},a_{j+m}) \in K$ (and thus $\chi(a_j,j) \chi(a_{j+m},j+m) \in N_j$)  for every $j\in \{0,\ldots, m-1\}$. Fix such an integer $j$ and let $a=a_{j}$ and $b=a_{j+m}$. Since $(1,a)\in K$ and since $h$ normalises $K$, it follows that $(1,a)^{h^{-1}} \in K$. However, $(1,a)^{h^{-1}}=(a,1)$, implying that $(a,b) = (a,1)(1,b) \in K$, as required. We have thus shown that $\alpha \in M \cap N$.

To summarise, the intersection $M\cap N$ consists precisely of those $\alpha$ from (\ref{eqn:alpha}) for which
$(1,a_j) \in K$ for all $j\in \{0,\ldots,2m-1\}$. In particular, 
$$
 |M\cap N| = |\{b \mid (1,b) \in K\}|^{2m} = |L_{(A)}|^{2m},
$$ 
which implies that
$$
|\langle M, N \rangle| = |M| |N| / |M\cap N| = |L_{(A)}|^{2\ell m} |K|^m / |L_{(A)}|^{2m} = |L_{(A)}|^{2 m (\ell -1)} (|L|/2)^m.
$$
Thus:
\begin{equation}
\label{eqn:|G|}
|G| = |\D_{2\ell m}| |\langle M, N \rangle| = 4\ell m  |L_{(A)}|^{2 m (\ell -1)} (|L|/2)^m.
\end{equation}
\medskip

We shall now turn our attention to the graph  $\Gamma$. Recall that $\Gamma=\C(k,2\ell m,m-1)$ and that the vertices of $\Gamma$ are the traversing paths of length $m-2$ in $\Lambda$; that is:
$$\V(\Gamma)=\{ (u_1,i)(u_2,i+1)\ldots (u_{m-1},i+m-2) \mid (u_1, \ldots, u_{m-1}) \in(\ZZ_k)^{m-1},i\in\ZZ_{2\ell m}\}.$$
Further, recall that since every automorphism of $\Lambda$ preserves the set of all such paths, there exists a natural faithful action of $G$ as a group of automorphisms of $\Gamma$. Moreover, observe that for any fixed $i\in \ZZ_{2\ell m}$, the group $N$ acts transitively on the set $\V_i = \{(u_1,i)(u_2,i+1)\ldots (u_{m-1},i+m-2) \mid (u_1, \ldots, u_{m-1}) \in(\ZZ_k)^{m-1}\}$, and that the group $\la \sigma \ra$ cyclically permutes the family $\{ \V_i \mid i\in \ZZ_{2\ell m}\}$. Since the latter family is a partition of $\V(\Gamma)$, this shows that $\la N, \sigma \ra$ (and thus $G$) is transitive on $\V(\Gamma)$.

Let  $v=(0,1)(0,2)\ldots(0,m-1)$ be a traversing path in $\Lambda$ of length $m-2$, interpreted as a vertex of $\Gamma$. Observe that the neighbourhood $\Gamma(v)$ decomposes into a disjoint union of subsets
\begin{eqnarray*}
\Gamma^-(v) & = & \{(j,0)(0,1)\ldots(0,m-2)\mid j\in\ZZ_k\}, \\
\Gamma^+(v) & = & \{(0,2)\ldots(0,m-1)(j,m)\mid j\in\ZZ_k\},
 \end{eqnarray*}
with $\Gamma^-(v)$ and $\Gamma^+(v)$ being blocks of imprimitivity for $G_v$. Hence there is a natural identification of the sets $\Gamma^-(v)$ and $\Gamma^+(v)$ with the sets $\ZZ_k \times \{0\}$ and $\ZZ_k \times \{m\}$, respectively. This induces an action of $G_v$ on $\ZZ_k\times\{0,m\}$.
 
Recall that the sets $A$ and $B$ were identified with the sets $\ZZ_k\times \{0\}$ and $\ZZ_k\times \{1\} $, which gives rise to a further identification of $\Gamma^-(v)$ with $A$ and $\Gamma^+(v)$ with $B$ (where $(x,1) \in B$ is identified with $(x,m) \in \Gamma^+(v)$). In particular, $G_v$ can be viewed as acting on $A\cup B$. Since $\Gamma^-(v)$ and $\Gamma^+(v)$ are blocks of imprimitivity for $G_v$, so are $A$ and $B$.

Observe that an element of $G_v$, when viewed as acting on $A\cup B$, preserves each of $A$ and $B$ setwise if and only if, when viewed as acting on $\V(\Lambda)$, it preserves each of the sets $\{i\} \times \ZZ_{2\ell m} \in \cP$. Now recall that the kernel of the action of $G$ on $\cP$ is $\la M,N\ra$ and hence the kernel of the action of $G_v$ on the partition $\{A,B\}$ of $A \cup B$ is $\la M,N\ra_v$. 

We shall now prove that $G_v^{\Gamma(v)}$ is permutation isomorphic to $L$.   We begin by proving that $M_v^{\Gamma(v)} \le N_v^{\Gamma(v)}$. Note that by (\ref{eq:N}), it suffices to prove that $(M_i)_v^{\Gamma(v)} \le N_v^{\Gamma(v)}$ for all $i\in \{0,\ldots,2\ell m - 1\}$. Clearly, if $i\not\in \{0,m\}$, then $(M_i)_v$ acts trivially on $\ZZ_k\times \{0,m\}$ and hence on $\Gamma(v)$. Further, observe that both $M_0$ as well as $M_m$ fix the vertex $v$. Recall that a typical element of $M_0$ is of the form $[b]_0$ for $b\in \Sym_k$ such that $(1,b) \in K$. Since $K$ is normal in $L$, it follows that $(1,b)^{h^{-1}} \in K$. However, $(1,b)^{h^{-1}} = (b,1)$. It follows that the element $\chi(b,0)\chi(1,m)$ is in $N_0$. However, the latter element clearly fixes $v$ and induces the same permutation on $\Gamma(v)$ as the element $[b]_0$ of $M_0$ does. In particular, $(M_0)^{\Gamma(v)} \le N_v^{\Gamma(v)}$. A similar computation shows that  $(M_m)^{\Gamma(v)} \le N_v^{\Gamma(v)}$, and therefore, that $M_v^{\Gamma(v)} \le N_v^{\Gamma(v)}$.

We shall now show that $N_v^{\Gamma(v)}$ is permutation isomorphic to $K$. Recall that  $N=N_0 \times \cdots \times N_{m-1}$ and note that, for $i\not = 0$, the subgroup $(N_i)_v$  acts trivially on $ \ZZ_k\times\{0,m\} $, and hence on  $\Gamma(v)$. Moreover, $(N_0)_v=N_0$, implying that $N_v^{\Gamma(v)} = (N_0)^{\Gamma(v)}$. Since $N_0$ consists of the elements $\chi(a,0)\chi(b,m)$ for $(a,b) \in K$, the identification of $\Gamma(v)$ with $\ZZ_k\times \{0,m\}$ and also with $A\cup B$ clearly implies that $(N_0)^{\Gamma(v)}$ is permutation isomorphic to $K$ and hence $N_v^{\Gamma(v)}$ is permutation isomorphic to $K$, as claimed.

Since $M_v^{\Gamma(v)} \le N_v^{\Gamma(v)}$, this implies that $\la M,N\ra_v^{\Gamma(v)}$ is permutation isomorphic to $K$. Now recall that $\la M,N\ra_v^{\Gamma(v)}$ is in fact the kernel of the action of $G_v^{\Gamma(v)}$ on the partition of $\{A,B\}$ of $A\cup B$ (after the usual identification of $\Gamma(v)$ with $A\cup B$).
 To conclude the proof that $G_v^{\Gamma(v)}$ is permutation isomorphic to $L$, it thus suffices to exhibit an element of $G_v$ which acts on $A\cup B$ as the permutation $h$.

By~$(\ref{eq:sigmapower})$, we have $\sigma^m=\chi(r,-m)\ldots \chi(r,-1)s^m$. By~(\ref{eqn:tau}), it follows that 
\begin{equation}\label{tausigma}
\tau^{-1}\sigma^m=\chi(r,m)ts^m.
\end{equation}
Clearly, both $\chi(r,m)$ and $ts^m$ fix $v$ and hence so does $\tau^{-1}\sigma^m$. Let us now compute the permutation induced by $\tau^{-1}\sigma^m$ on $A \cup B$. Let $(j,1)$ be an arbitrary element of $B$ and recall that this element is represented by $(0,2)\ldots(0,m-1)(j,m)\in\Gamma^+(v)$. Now, note that
\begin{eqnarray*}
\bigl((0,2)\ldots(0,m-1)(j,m)\bigr)^{\tau^{-1}\sigma^m}&(\ref{tausigma})\atop=&\bigl((0,2)\ldots(0,m-1)(j,m)\bigr)^{\chi(r,m)ts^m}\\
&(\ref{chidef})\atop=&\bigl((0,2)\ldots(0,m-1)(j^r,m)\bigr)^{ts^m}\\
&=&\bigl((j^r,0)(0,1)\ldots(0,m-2)\bigr).
\end{eqnarray*} 
This shows that $(j,1)^{\tau^{-1}\sigma^m}=(j^r,0)\in A$.

Similarly,  let $(j,0)$ be an arbitrary element of $A$. This element is represented by $(j,0)(0,1)\ldots(0,m-2)\in\Gamma^-(v)$ and an analogous computation yields:
$$\bigl((j,0)(0,1)\ldots(0,m-2)\bigr)^{\tau^{-1}\sigma^m}=(0,2)\ldots(0,m-1)(j,m).$$
This shows that $(j,0)^{\tau^{-1}\sigma^m}=(j,1)\in B$. By (\ref{eqn:h}), this shows that $\tau^{-1}\sigma^m$, viewed as a permutation on $A\cup B$, is equal to $h$. This concludes the proof of the fact that $G_v^{\Gamma(v)}$ is permutation isomorphic to $L$. 

Since $\Gamma$ is $G$-vertex-transitive and $L$ is transitive, it follows that $\Gamma$ is $G$-arc-transitive. It remains to determine the order of the arc-stabiliser $G_{uv}$ of an arc $uv$ of $\Gamma$. Since $|\V(\Gamma)| = 2\ell m k^{m-1}$ and the valence of $\Gamma$ is $2k$, it follows that $|\A(\Gamma)| = 4\ell m k^{m}$. Further, since $L$ is a transitive permutation group of degree $2k$, it follows that $|L|/2k=|L_\omega|$, where $L_\omega$ is a point-stabiliser in $L$. Using (\ref{eqn:|G|}) for the order of $G$, we thus get the following:

$$|G_{uv}| = \frac{4\ell m  |L_{(A)}|^{2 m (\ell -1)} (|L|/2)^m }{4 \ell mk^{m}} =|L_{(A)}|^{2 m (\ell -1)} |L_\omega|^m.$$
\end{proof}

\begin{corollary}\label{cor:cons}
Let $L$ be a transitive permutation group of degree $2k$ with a system of imprimitivity consisting of two blocks $A$ and $B$, let $L_\omega$ be a point-stabiliser in $L$ and let $L_{(A)}$ be the pointwise stabiliser of $A$ in $L$. Then $L$ is $f$-graph-unrestrictive where $f(n)=|L_{(A)}|^{\frac{n}{k}-4} |L_\omega|^2$. In particular, if $L_{(A)}\neq 1$, then $L$ has graph-type $\Exp$.

Moreover, if  $L_\omega \not = 1$, then $L$ is $\Pol$-graph-unrestrictive; in fact, $L$ is $n^{\alpha}$-graph-unrestrictive for every $\alpha$ with $0<\alpha < \frac{\log |L_\omega|}{\log k}$. 
\end{corollary}

\begin{proof}
For $\ell\ge 1$, let $\Gamma_\ell = \C(k,4\ell,1)$. By Theorem~\ref{theo:cons} (applied with $m=2$), there exists $G_\ell \le \Aut(\Gamma_\ell)$ such that $(\Gamma_\ell,G_\ell)$ is locally-$L$ and $|(G_\ell)_{uv}| =|L_{(A)}|^{4 (\ell -1)} |L_\omega|^2$. Since $|\V(\Gamma_\ell)|=4k\ell$, the latter is equal to $|L_{(A)}|^{\frac{|\V(\Gamma_\ell)|}{k} -4} |L_\omega|^2$. This shows that $L$ is $f$-graph-unrestrictive where $f(n)=|L_{(A)}|^{\frac{n}{k}-4} |L_\omega|^2$, which is a function in the class $\Exp$ provided that $L_{(A)} \not =1$.

Now suppose that $L_\omega\neq 1$. For $m\geq 2$, let $\Gamma_m =\C(k,2m,m-1)$. By Theorem~\ref{theo:cons} (applied with $\ell = 1$),
there exists $G_m \le \Aut(\Gamma_m)$ such that $(\Gamma_m,G_m)$ is locally-$L$ and $|(G_m)_{uv}| =|L_\omega|^m$. 

Let $n=|\V(\Gamma_m)|$ and observe that $n= 2mk^{m-1}$. Let $c=|L_\omega|$ and fix  $\alpha$ such that $0<\alpha < \frac{\log c}{\log k}$.  Then there exists $\epsilon >0$ such $\alpha = \frac{\log c}{\log (k+\epsilon)}$. Furthermore, since $\frac{n}{(k+\epsilon)^m}=\frac{2mk^{m-1}}{(k+\epsilon)^m} \to 0$ as $m\to \infty$, there exists $m_0$ such that for every $m>m_0$, we have $n < (k+\epsilon)^m$, and thus $\frac{\log n}{\log(k+\epsilon)} < m$.
Hence
$$
 n^\alpha = n^{\frac{\log c}{\log (k+\epsilon)}} = c^{\frac{\log n}{\log (k+\epsilon)}} < c^m = |G_{uv}|.
$$
This proves that $L$ is $f$-graph-unrestrictive for $f(n) = n^\alpha$, and, in particular, that $L$ is $\Pol$-graph-unrestrictive.

\end{proof}

\section{Imprimitive groups of degree $6$ that do not admit a system of imprimitivity consisting of two blocks.}
\label{deg6}

The results proved so far in this paper together with previously known results are enough to settle  Problem~\ref{mainprob} for transitive permutation groups of degree at most $7$ with the exception of three groups of degree $6$ that are imprimitive but do not admit a system of imprimitivity consisting of two blocks of size $3$ (see Section~\ref{sec:smalldegree} for details). In this section, we will show that these three groups are $\Subexp$-graph-unrestrictive.

By~\cite{conway}, there are five transitive groups of degree $6$ that are imprimitive but do not admit a system of imprimitivity consisting of two blocks. Using the taxonomy of \cite{conway}, they are $A_4(6)$, $2A_4(6)$, $S_4(6d)$, $S_4(6c)$ and $2S_4(6)$. Note that $2S_4(6)$ is in fact isomorphic to the wreath product $\ZZ_2\wr S_3$ in its imprimitive action on $6$ points, while $A_4(6)$ is isomorphic to $\ZZ_2^2\rtimes C_3$ viewed as a subgroup of index $2$ in the wreath product $\ZZ_2\wr C_3$. The group $A_4(6)$ is thus a normal subgroup of index $4$ in $2S_4(6)$ and $2S_4(6)/A_4(6) \cong \ZZ_2^2$. The remaining three groups $2A_4(6)$, $S_4(6d)$ and $S_4(6c)$ are precisely the three groups $L$ with the property that $A_4(6) < L < 2S_4(6)$. The main result of this section is the following:
\begin{theorem}
\label{thm:deg6}
Let $L$ be a transitive permutation group of degree $6$ that is imprimitive but does not admit a system of imprimitivity consisting of two blocks. Then $L$ is $f$-graph-unrestrictive where $f(n)=4^{\sqrt{\frac{n}{3}}-1}$.
\end{theorem}

It will be convenient for us to describe the five groups $L$ satisfying $A_4(6)\le L \le 2S_4(6)$ as subgroups of the symmetric group on the set $\Omega=\{0,1,\ldots, 5\}$ in terms of generators. In what follows, we shall follow the notation introduced in \cite{conway} as closely as possible. Let $a,b,e,f$ be the following permutations of $\Omega$:
\begin{equation}
\label{eqn:abef}
 a = (0~2~4)(1~3~5),\quad
b = (1~5)(2~4),\quad
e = (1~4)(2~5),\quad
f = (0~3).
\end{equation}
Then by \cite[Appendix A]{conway}, the groups $A_4(6)$ and $2S_4(6)$ can be expressed as:
\begin{equation}
\label{eqn:LL}
A_4(6) = \la a, e \ra, \quad
\hbox{ and } \quad
2S_4(6) = \la f,a,b \ra.
\end{equation}

The rest of the section is devoted to the proof of Theorem~\ref{thm:deg6}. In this endeavour, 
the following lemma will prove very useful. (For a graph $\Gamma$, a vertex $v\in \V(\Gamma)$ and a group $G\le \Aut(\Gamma)$,
let $G_v^{[1]}$ denote the kernel of the action of $G_v$ on $\Gamma(v)$.)

\begin{lemma}
\label{lem:s}
Let $\Lambda$ be a graph, let $v$ be a vertex of $\Lambda$ and let $A$ and $B$ be vertex-transitive groups of automorphisms of $\Lambda$ such that $B\trianglelefteq A$. Then, for every permutation group $L$ such that $B_v^{\Lambda(v)}\leq L\leq A_v^{\Lambda(v)}$, there exists $C$ such that $B\leq C\leq A$,
$C_v^{\Lambda(v)}=L$ and $C_v^{[1]} = A_v^{[1]}$.
\end{lemma}

\begin{proof}
Let $\pi\colon A_v \to A_v^{\Lambda(v)}$ be the epimorphism that maps each $g\in A_v$ to the permutation induced by $g$ on $\Lambda(v)$. Observe that $B_v\leq \pi^{-1}(L)\leq A_v$ and $(\pi^{-1}(L))^{\Lambda(v)}=\pi(\pi^{-1}(L))=L$. Finally, let $C=B\pi^{-1}(L)$. Clearly, $B\leq C\leq A$ and  $C_v=B\pi^{-1}(L)\cap A_v$. By the modular law, the latter equals  $(B\cap A_v)\pi^{-1}(L)$ and hence $C_v=\pi^{-1}(L)$. In particular, the group $A_v^{[1]}$ (being the kernel of $\pi$) is contained in 
$C_v$ and thus is equal to $C_v^{[1]}$. Moreover, $C_v^{\Lambda(v)}= \pi(C_v) = \pi(\pi^{-1}(L)) = L$.
\end{proof}

The proof of Theorem~\ref{thm:deg6} makes use of certain cubic arc-transitive graphs having a large nullity over the field $\FF_2$
of order $2$, the existence of which was proved in \cite{null}. Let us explain this in more detail.

Let $\Gamma$ be a graph with vertex-set $V$, let $\FF$ be a field and let $\FF^V$ be the set of all functions from $V$ to $\FF$,
viewed as an $\FF$-vector space (with addition and multiplication by scalars defined pointwise). The {\em $\FF$-nullspace} of
$\Gamma$ is then the set of all elements $x\in \FF^V$ such that for every $v\in V$ we have
$$%
 \sum_{u\in \Gamma(v)}x(u) = 0.
$$

The $\FF$-nullspace is clearly a subspace of the vector space $\FF^V$ and its dimension over $\FF$ is called the {\em $\FF$-nullity} of $\Gamma$. We can now state and prove the crucial step in the proof of Theorem~\ref{thm:deg6}.

\begin{theorem}
\label{thm:L}
Let $L$ be one of the five imprimitive permutation groups of degree $6$ admitting no blocks of imprimitivity of size $3$.
Let $\Gamma$ be a connected cubic graph, let $M$ be the  $\FF_2$-nullspace of $\Gamma$ and suppose that $M$ is non-trivial.
Suppose further that $\Aut(\Gamma)$ contains two subgroups $H$ and $G$ with $H$ acting regularly on the set of arcs of $\Gamma$, with $G$ acting regularly on the set of $2$-arcs of $\Gamma$ and with $H\leq G$. Let $\Lambda=\Gamma[2K_1]$ be the lexicographic product of $\Gamma$ with the edgeless graph $2K_1$ on two vertices. Then  there exists an arc-transitive subgroup $C\le \Aut(\Lambda)$, such that the pair $(\Lambda,C)$ is locally-$L$ and the stabiliser of an arc of $\Lambda$ in $C$ has order $|M| |L|/48\geq |M|/4$.
\end{theorem}

\begin{proof}
Let $V$ denote the vertex-set of $\Gamma$. We will think of the vertex-set of $\Lambda$ as $\FF_2\times V$. Observe that there is a natural embedding of $G$ as well as of the additive group of $\FF_2^V$ into $\Aut(\Lambda)$ given by
$$%
 (a,v)^g = (a,v^g)\>\> \hbox{ and }\>\> (a,v)^x = (a+x(v),v)
$$%
for every $(a,v) \in \FF_2\times V$, $g\in G$ and $x\in \FF_2^V$. In this sense, we may view $G$ and $\FF_2^V$ as subgroups of $\Aut(\Lambda)$. 

Let $\one\in \FF_2^V$ be the constant function mapping each vertex to $1$ and let $N=\la \one \ra \oplus M\leq \FF_2^V\leq \Aut(\Lambda)$. Finally, let $A$
and $B$ be the subgroups of $\Aut(\Lambda)$ defined by
$$A= \la N,G\ra  \qquad \hbox{ and } \qquad B = \la M,H\ra.$$

Since $M$ is the $\FF_2$-nullspace of $\Gamma$ and $G\leq\Aut(\Gamma)$, it follows that $M$ is normalised by $G$ and hence by $H$. Moreover,  the automorphism $\one$ of $\Lambda$ is clearly centralised by both $M$ and $G$ and hence $N$ is also normalised by $G$. In fact, since $M$ and $N$ act trivially on the partition $\{\FF_2 \times \{v\} : v \in V\}$ of $\V(\Lambda)$, while $G$ and $H$ act faithfully, we see that $N\cap G = M\cap H = 1$, and thus

\begin{equation}
\label{eqn:AB}
A = N\rtimes G \hbox{ and } B = M \rtimes H.
\end{equation}

Let us now show that $B\norml A$. It is clear from the definition that $B\leq A$. We have already noted that $G$ normalises $M$. Moreover, $H$ has index $2$ in $G$ and hence is normal in $G$. It follows that $G$ normalises $B = M \rtimes H$.  Note that $A$ is generated by $B$, $G$ and $\one$. Since $\one$ centralises $B$, this implies that $B\norml A$.

Let $v$ be a vertex of $\Gamma$ and let $\tv=(0,v)$ be the corresponding vertex of $\Lambda$.
Let us now show that
\begin{equation}
\label{eqn:local}
      A^{\Lambda(\tv)}_{\tv}=N^{\Lambda(\tv)}_{\tv}G^{\Lambda(\tv)}_{\tv} \quad \hbox{ and } \quad
      B^{\Lambda(\tv)}_{\tv}=M^{\Lambda(\tv)}_{\tv}H^{\Lambda(\tv)}_{\tv}.
\end{equation}

Let $a\in A$. By (\ref{eqn:AB}), $a=xg$ for some $x\in N$ and $g\in G$. Then $a\in A_{\tv}$ if and only if $(0,v)=(0,v)^a=(x(v),v^g)$ if and only if $g\in G_v$ and $x\in N_{\tv}$. Note that an element $g\in G$ fixes $v\in\V(\Gamma)$ if and only if (when viewed as an automorphism of $\Lambda$) it fixes $\tv$. In this sense we can write $G_v=G_{\tv}$. We have thus shown that  $A_{\tv}=N_{\tv}G_{\tv}$. 

On the other hand, $B_{\tv}=A_{\tv}\cap B=N_{\tv}G_{\tv}\cap MH=M_{\tv}H_{\tv}$. To conclude the proof of (\ref{eqn:local}),
 apply the epimorphism that maps each $a\in A_{\tv}$ to the permutation induced by $a$ on $\Lambda(\tv)$ 
 to the equalities $A_{\tv}=N_{\tv}G_{\tv}$ and $B_{\tv}=M_{\tv}H_{\tv}$.

We shall now prove that the permutation groups 
$B^{\Lambda(\tv)}_{\tv}$ and $A^{\Lambda(\tv)}_{\tv}$ are permutation isomorphic to the groups $A_4(6)$ and $2S_4(6)$, respectively.
Let $\Gamma(v)=\{s,t,u\}$ and label the six neighbours $(0,s),(1,s),(0,t),(1,t),(0,u),(1,u)$ of $\tv$ by $0,3,2,5,4,1$, respectively.
Since $H$ acts regularly on the set of the arcs of $\Gamma$, it follows that $H_v^{\Gamma(v)} = \la (s~t~u) \ra \cong C_3$. Similarly, since $G$ acts regularly on the set of the $2$-arcs of $\Gamma$, it follows that $G_v^{\Gamma(v)} = \la (s~t~u), (t~u)\ra \cong\Sym_3$. In particular, when $H_v$ and $G_v$ are viewed as acting on $\Lambda(\tv)$, we see that
 \begin{equation}
 \label{eqn:HG}
 H^{\Lambda(\tv)}_{\tv}=\langle a\rangle, \> \>
 G^{\Lambda(\tv)}_{\tv}=\langle a,b\rangle, \> \hbox{  where } \>
 a = (0~2~4)(1~3~5),\> \>
b = (1~5)(2~4).
 \end{equation}
Note that the permutations $a$ and $b$ above are the same as the permutations defined in (\ref{eqn:abef}). Let us define the following permutations of $\Lambda(\tv)$:
$$e= (1~4)(2~5),\> f=(0~3) \>\hbox{ and } \> k=(0~3)(1~4).$$
Note that $e$ and $f$ are as in (\ref{eqn:abef}) and that $k=e^a$. Moreover, note that $\la e,k\ra \cong \ZZ_2^2$ and that it is in fact the group consisting of all even permutations of $\la f,f^a, f^{a^2}\ra$. We shall now show that  
\begin{equation}
\label{eqn:MN}
M_\tv^{\Lambda(\tv)} = \la e,k \ra = \la e,e^a\ra,\> \hbox{ and }\> N_\tv^{\Lambda(\tv)}= \la f,f^a,f^{a^2}\ra. 
\end{equation}

 Since $M\neq 0$ and $\Gamma$ is connected and arc-transitive, there exists $x\in M$ such that $x(v)=0$ and $x(s)=1$ (in particular, $x\in M_\tv$). Since $x$ belongs to the $\FF_2$-nullspace of $\Gamma$, it follows that $x(t)+x(u)=1$. Let $\gamma$ be the permutation of $\Lambda(\tv)$ induced by the action of $x$ on $\Lambda(\tv)$. Since $M_{\tv}$ is normalised by $G_{\tv}$, it follows that $M_{\tv}^{\Lambda(\tv)}$ is normalised by $G_{\tv}^{\Lambda(\tv)}$ and hence $\gamma,\gamma^a,\gamma^{a^2}\in M_{\tv}$. Observe that $\{\gamma,\gamma^a,\gamma^{a^2}\}=\{e,k,ek\}$, and thus $\la e, k \ra \le M_\tv^{\Lambda(\tv)}$. On the other hand, since $M$ is the $\FF_2$-nullity of $\Gamma$, it follows that $|M_\tv^{\Lambda(\tv)}| \le 4$, and thus $M_{\tv}^{\Lambda(\tv)} = \langle e,k\rangle$, as claimed.

Since $\Gamma$ is $G$-arc-transitive, there exists $g\in G$ such that $(s,v)^g=(v,s)$. Let $y=\one+x+x^g$. Then $y(v)=1+x(v)+x(s)=1+0+1=0$ and hence $y\in N_{\tv}$. Moreover, since $x,x^g\in M$ it follows that $x(s)+x(t)+x(u)=0=x^g(s)+x^g(t)+x^g(u)$ and therefore $y(s)+y(t)+y(u)=1$. It follows that $y^{\Lambda(\tv)}\in N_{\tv}^{\Lambda(\tv)} \setminus M_{\tv}^{\Lambda(\tv)}$ and hence $N_{\tv}^{\Lambda(\tv)}$ has order at least $8$. On the other hand, $N_{\tv}^{\Lambda(\tv)}$ acts trivially on the partition $\{ \{0,3\}, \{1,4\}, \{2,5\}\}$ of $\Lambda(\tv)$, and is therefore a subgroup of  $\la f,f^a,f^{a^2}\ra \cong \ZZ_2^3$. This proves that $N_{\tv}^{\Lambda(\tv)} = \la f,f^a,f^{a^2}\ra$, as claimed.

Finally, if we combine (\ref{eqn:local}), (\ref{eqn:HG}) and (\ref{eqn:MN})  we obtain
$A^{\Lambda(\tv)}_{\tv}=$ $N^{\Lambda(\tv)}_{\tv}G^{\Lambda(\tv)}_{\tv} =$ $\la f,f^a,f^{a^2},a,b\ra = \la f,a,b \ra$ and
$B^{\Lambda(\tv)}_{\tv}=$ $M^{\Lambda(\tv)}_{\tv}H^{\Lambda(\tv)}_{\tv} = $ $\la e, e^a, a\ra = \la a,e\ra$.
Hence, by (\ref{eqn:LL}), we see that

\begin{equation}
\label{eqn:ABLL}
      A^{\Lambda(\tv)}_{\tv}= 2S_4(6) \quad \hbox{ and } \quad
      B^{\Lambda(\tv)}_{\tv} = A_4(6).
\end{equation}      

In particular, this shows that $B^{\Lambda(\tv)}_{\tv} \le L \le A^{\Lambda(\tv)}_{\tv}$. Since $B$ is normal in $A$, Lemma~\ref{lem:s} implies that there exists
$C\le \Aut(\Lambda)$ with $B\le C\le A$, $A_\tv^{[1]} = C_{\tv}^{[1]}$ and such that the pair $(\Lambda,C)$ is locally-$L$. 

Since $A$ and $C$ are both transitive on $\V(\Lambda)$ and since $A_\tv^{[1]} = C_{\tv}^{[1]}$, it follows that 
$|A:C| =$  $|A_v:C_v| =$ $|A^{\Lambda(\tv)}_{\tv} : C^{\Lambda(\tv)}_{\tv}|=$ $|2S_4(6):L|=48/|L|$. By (\ref{eqn:AB}) we see that $|A| = |N| |G| = 2|M||G|$. Moreover, since $G$ is transitive on $\V(\Gamma)$ and $|G_v|=6$, it follows that $|A| = 12|M||\V(\Gamma)|=6|M||\V(\Lambda)|$. Hence $|C| = |A| /  |A:C|= |M||V(\Lambda)||L|/8$. Since $C$ is transitive on the $6|\V(\Lambda)|$ arcs of $\Lambda$, this implies
that the stabiliser of an arc of $\Lambda$ in $C$ has order $|M||L|/48$. Finally, since $|L|\geq 12$, we have $|M||L|/48\geq |M|/4$.
\end{proof}

In order to apply Theorem~\ref{thm:L} to prove Theorem~\ref{thm:deg6}, we need to establish the existence of
appropriate cubic arc-transitive graphs with large $\FF_2$-nullity. The existence of such graphs follows immediately from \cite[Construction 8, Proposition 9]{null}, where the following was proved:

\begin{theorem}
\label{lem:null}
For every positive integer $m$ there exists a connected cubic graph of order $6\cdot2^{2m}$ and $\FF_2$-nullity at least $4\cdot2^{m}$
which admits an arc-regular group of automorphisms $H$ and a $2$-arc-regular group of automorphisms $G$ such that $H\le G$.
\end{theorem}

We can now combine Theorem~\ref{thm:L} and Theorem~\ref{lem:null} to prove Theorem~\ref{thm:deg6}.\\

\noindent\emph{Proof of Theorem~\ref{thm:deg6}.} Let $L$ be one of the five imprimitive permutation groups of degree $6$ admitting no system of imprimitivity with blocks of size $3$. Let $m$ be a positive integer. By Theorem~\ref{lem:null}, there exists a connected cubic graph $\Gamma_m$ of order $6\cdot2^{2m}$ and $\FF_2$-nullity at least $4\cdot 2^{m}$ which admits an arc-regular group of automorphisms $H_m$ and a $2$-arc-regular group of automorphisms $G_m$ such that $H_m\le G_m$. Let $\Lambda_m=\Gamma_m[2K_1]$ and let $n=|\V(\Lambda_m)|=12\cdot2^{2m}$. By Theorem~\ref{thm:L}, there exists an arc-transitive subgroup $C_m\le \Aut(\Lambda_m)$, such that the pair $(\Lambda_m,C_m)$ is locally-$L$ and the stabiliser of an arc of $\Lambda_m$ in $C_m$ has order at least $\frac{|M|}{4}=4^{2\cdot 2^{m}-1}=4^{2\cdot \sqrt{\frac{n}{12}}-1}=4^{\sqrt{\frac{n}{3}}-1}$. Thus $L$ is $f$-graph-unrestrictive where $f(n)=4^{\sqrt{\frac{n}{3}}-1}$. \hfill \qedsymbol

\section{Transitive permutation groups of degree at most $7$}
\label{sec:smalldegree}

In this section, we determine the status of Problem~\ref{mainprob} for $L$ a transitive permutation group of degree at most $7$. The results are summarised in Table~\ref{table:degree7}, where we use the notation from~\cite{conway}.

Since $L$ has degree at most $7$, it follows that $L$ is graph-restrictive if and only if it is regular or primitive (see~\cite[Proposition~$14$]{PSVRestrictive}). In this case, $L$ has graph-type $\Cons$. We may thus assume that $L$ is neither regular nor primitive. Since permutation groups of prime degree are primitive, it thus suffices to consider the cases when $L$ has degree $4$ or $6$. 

The only transitive permutation group of degree $4$ which is neither regular nor primitive is $\D_4$. This group is permutation isomorphic to the imprimitive wreath product $\ZZ_2\wr\ZZ_2$ and, by Theorem~\ref{thm:wreath}, has graph-type $\Exp$. 

We now consider the case when $L$ has degree $6$. Suppose that $L$ admits a system of imprimitivity consisting of two blocks of size $3$ and let $A$ be one of these blocks. If $L_{(A)}\neq 1$, then by Corollary~\ref{cor:cons} $L$ has graph-type $\Exp$. If $L_{(A)}=1$, then it is easily checked that $L$ must be permutation isomorphic to $\D_6$. It follows by Corollary~\ref{cor:cons} that $L$ is $\Pol$-graph-unrestrictive. On the other hand, it follows from~\cite[Theorem A]{Verret2} that $\D_6$ is $\Pol$-graph-restrictive and hence $\D_6$ has graph-type $\Pol$.

\begin{scriptsize}
\begin{table}[hhh!]
\begin{center}
\begin{tabular}{|p{90mm}|c|c|}\hline
 & No. of grps & Graph-type  \\\hline
Regular or primitive &26  & $\Cons$ \\\hline
$\D_6$ & 1 & $\Pol$  \\\hline
$\D_4$, $\ZZ_3\wr\ZZ_2$, $\ZZ_2\wr\ZZ_3$, $\Sym_3\wr\ZZ_2$, $\ZZ_2\wr\Sym_3$, $F_{18}(6)\colon 2$, $F_{36}(6)$ & 7 & $\Exp$  \\\hline
$A_4(6)$, $S_4(6c)$, $S_4(6d)$ & 3 & $\Subexp$ or $\Exp$  \\\hline
\end{tabular}
\medskip
\caption{\scriptsize{Graph-types of transitive permutation groups of degree at most $7$}}\label{table:degree7}
\end{center}
\end{table}
\end{scriptsize}

It thus remains to deal with transitive permutation groups of degree $6$ which are imprimitive but do not admit a system of imprimitivity consisting of two blocks of size $3$. As we saw in Section~\ref{deg6}, there are five such groups, two of them are wreath products and hence have graph-type $\Exp$. As for the remaining three, it follows from Theorem~\ref{thm:deg6} that they are $\Subexp$-graph-unrestrictive. Unfortunately, for none of these three groups we were able to decide whether it has graph-type $\Subexp$ or $\Exp$.

\end{document}